\documentclass[12pt,a4paper,oneside]{amsart}
\pdfoutput=1
\usepackage{amssymb}
\usepackage{esint}
\usepackage{mathtools}
\usepackage{hyperref}

\usepackage[english]{babel}
\usepackage[latin1]{inputenc}
\usepackage[pdftex]{graphicx}

\theoremstyle{plain}
\newtheorem{theorem}{Theorem}
\newtheorem{lemma}[theorem]{Lemma}

\theoremstyle{definition}
\newtheorem{definition}[theorem]{Definition}

\theoremstyle{remark}
\newtheorem{remark}[theorem]{Remark}

\newcommand{\T}{\mathbb{T}}
\newcommand{\R}{\mathbb{R}}
\newcommand{\Z}{\mathbb{Z}}
\newcommand{\C}{\mathbb{C}}
\newcommand{\N}{\mathbb{N}}

\newcommand{\ip}[2]{\left\langle#1,#2\right\rangle}

\newcommand{\der}{\mathrm{d}}

\renewcommand{\phi}{\varphi}
\newcommand{\abs}[1]{\left| #1 \right|}
\newcommand{\aabs}[1]{\left\| #1 \right\|}

\newcommand{\dummy}{{\,\cdot\,}}

\DeclareMathOperator{\Aut}{Aut}
\newcommand{\pol}[2]{\mathcal{P}_{#1}^{#2}}

\title{Tensor tomography in periodic slabs}
\author{Joonas Ilmavirta \and Gunther Uhlmann}
\thanks{Department of Mathematics and Statistics, University of Jyv\"askyl\"a, P.O. Box 35 (MaD) FI-40014 University of Jyv\"askyl\"a, Finland; \texttt{joonas.ilmavirta@jyu.fi}.}
\thanks{Department of Mathematics, University of Washington, Seattle, WA 98195-4350, USA;
Institute for Advanced Study of the Hong Kong University of Science and Technology, Hong Kong, China;
Department of Mathematics and Statistics, University of Helsinki, Finland; \texttt{gunther@math.washington.edu}}
\date{\today}

\begin{document}

\begin{abstract}
The X-ray transform on the periodic slab $[0,1]\times\mathbb T^n$, $n\geq0$, has a non-trivial kernel due to the symmetry of the manifold and presence of trapped geodesics. For tensor fields gauge freedom increases the kernel further, and the X-ray transform is not solenoidally injective unless $n=0$. We characterize the kernel of the geodesic X-ray transform for $L^2$-regular $m$-tensors for any $m\geq0$. The characterization extends to more general manifolds, twisted slabs, including the M\"obius strip as the simplest example.
\end{abstract}



\maketitle

\section{Introduction}

We study geodesic X-ray tomography of tensor fields on the manifold $M=[0,1]\times\T^n$, where $\T^n=\R^n/\Z^n$ and $n\geq0$.
This class of manifolds includes the interval~$[0,1]$ and the strip $[0,1]\times\T^1$.
We only consider geodesics joining boundary points, excluding trapped geodesics.

Our main result is theorem~\ref{thm:torus} which completely characterizes the kernel of the X-ray transform on~$M$ for tensor fields of any order.
The kernel is the sum of two kinds of functions: those arising from potentials (symmetrized differentials of tensor fields of lower order) and those depending only on the variable on~$[0,1]$.
See section~\ref{sec:slab} for details.
That is, both gauge freedom and symmetry cause kernel.
We are not aware of earlier observations --- and, in particular, characterizations --- of this kind of kernel.
In particular, the X-ray transform is always non-injective, even for scalar fields in all dimensions.

It is easy to see that the same kernel is present in the infinite slab $[0,1]\times\R^n$, but we do not pursue characterizing the kernel in that case.
Our result can be seen as the case for periodic~$L^2_{\text{loc}}$ tensor fields on $[0,1]\times\R^n$.
In practical problems where the slab is large but finite, we expect there to be a significant instability corresponding to the kernel of the infinite case.

Observe that if any decay or integrability conditions are imposed on scalar functions on $[0,1]\times\R^n$, then the obvious kernel vanishes.
The X-ray transform can be easily seen to be injective on compactly supported functions in $[0,1]\times\R^n$ using Helgason's support theorem for all straight lines avoiding the line $[0,1]\times\{0\}$.
A version of the support theorem for compactly supported~$L^1$ functions can be obtained through mollification, see e.g.~\cite[Proposition~5]{I:disk}.

The result can also be extended to broken ray tomography on the slab, where one of the surfaces of~$M$ reflects rays.
This can be achieved with a simple reflection argument; cf.~\cite{I:refl,H:brt-flat-refl,H:square}.

Our result can also be extended to a broader class of manifolds.
The slab~$M$ can be obtained by identifying some opposite faces of $[0,1]^{1+n}$.
If the gluing is done in a more exotic way, one ends up with what we call a twisted slab.
The simplest example of a twisted slab is the M\"obius strip.
For more details, see theorem~\ref{thm:twisted} and section~\ref{sec:twisted}.

\begin{remark}
\label{rmk:stretch}
The periodic slab can be stretched in different ways.
The interval can be any~$[0,L]$ and we can divide~$\R^n$ by any lattice obtained from~$\Z^n$ by a linear bijection.
For simplicity we restrict ourselves to $L=1$ and the lattice~$\Z^n$.
The results can be generalized in an obvious way.
If the lattice is stretched differently in different directions, then there are fewer twisted slabs.
For the sake of clarity, we do not include such stretched slabs in the statements of our results; this generalization of our results is elementary.
\end{remark}

The problem studied here is similar to X-ray tomography on tori, which has been studied by Abouelaz and Rouvi\`ere~\cite{AR:radon-torus,A:torus-plane-radon} and the first author~\cite{I:torus}, including tensor tomography.
However, since we look at geodesics joining boundary points, our set of admissible curves is different.

For tensor tomography results on manifolds, and their applications, we refer to the review~\cite{PSU:tensor-survey}.
Inverse boundary value problems for PDEs have been considered previously in slab geometry (see eg.~\cite{SW:complex-spherical-optics,KLU:schrodinger-slab,LU:slab}), but we are unaware of any developments in X-ray tomography in this setting.
Due to the inaccessibility of geodesics parallel to the slab, our problem can be regarded as a form of limited angle tomography (see eg. the thesis~\cite{F:limited-angle-thesis} or~\cite[Section~3]{I:pseudo-riemann-xrt}).
Of previous results on ray transforms in product geometry we mention the examples and counterexamples listed in~\cite[Section~6]{I:refl}, and the recent result by Salo~\cite[Theorem~1.3]{S:calderon-normal-form} on injectivity of the ray transform on the product of non-closed manifolds.
The manifolds studied here are products of a manifold with boundary ($[0,1]$) and a closed manifold ($\T^n$) and therefore fall outside Salo's result.

\section{Tensor tomography in a periodic slab}
\label{sec:slab}

We consider tensor tomography on the manifold $M=[0,1]\times\T^n$ for any $n\geq0$.
Here $\T^n=\R^n/\Z^n$ denotes the flat torus of dimension~$n$.
The space~$M$ is equipped with the standard Euclidean metric, and functions on~$M$ can be regarded as functions on $[0,1]\times\R^n$ which are periodic --- or, equivalently, invariant under the translation action of~$\Z^n$ on~$\R^n$.

A symmetric covariant tensor field of order~$m$ can be identified with a function $f\colon M\times\R^{n+1}\to\C$, which is a homogeneous polynomial of order~$m$ in the second variable.
We will often write the variable on~$M$ as $(x,y)\in[0,1]\times\T^n$ and the variable of the polynomial as $(v,w)\in\R\times\R^n$.
The function~$f$ is then written as $f(x,y;v,w)$.
If $m=0$, then~$f$ is a scalar function and the polynomial is of order zero.

Any regularity assumptions on tensor fields are assumptions on the coefficients of the polynomial which are functions on~$M$.
We denote the space of homogeneous polynomials of order~$m$ in~$\R^{n+1}$ by~$\pol{m}{n+1}$, so that a tensor field of order~$m$ is a function $M\to\pol{m}{n+1}$.
We have naturally $\pol{0}{n+1}=\C$ and we also denote $\pol{-1}{n+1}=0$.

The X-ray transform~$I^mf$ of~$f$ encodes the integrals of~$f$ over all geodesics.
Geodesics through~$M$ can be parametrized by $a\in\T^n$ and $b\in\R^n$ so that correspond to the geodesic $[0,1]\ni t\mapsto(t,a+bt)\in[0,1]\times\T^n$.
This makes~$I^mf$ into a function $\T^n\times\R^n\to\C$.
The integral of~$f$ over this geodesic is
\begin{equation}
I^mf(a,b)
=
\int_0^1f(t,a+bt;1,b)\der t.
\end{equation}
This unusual scaling of velocity --- it has length $\sqrt{1+\abs{b}^2}$ instead of~$1$ --- is convenient in slab geometry.

The X-ray transform~$I^mf$ does not uniquely determine~$f$.
There are two obstructions:
If $f(x,y;v,w)$ is independent of~$y$ and integrates to zero over~$x$ for any fixed~$(v,w)$, then $I^mf=0$ but~$f$ may still be non-trivial.
If~$g$ is a tensor field of order $m-1$ vanishing at $\partial M=\{0,1\}\times\T^n$ and $f=\der g$, where~$\der$ is the symmetrized covariant derivative, then $I^mf=0$.

Our result is that these are the only obstructions:

\begin{theorem}
\label{thm:torus}
Let $m\geq0$ and $n\geq0$ be integers.
Let~$f$ be an $L^2$-regular tensor field on $M=[0,1]\times\T^n$.
In other words, $f\in L^2(M;\pol{m}{n+1})$.
There is a constant $C=C(n,m)$ so that the following are equivalent:
\begin{enumerate}
	\item The X-ray transform of~$f$ vanishes: $I^mf=0$.
	\item There are $h\in L^2([0,1],\pol{m}{n+1})$ and $g\in H^1_0(M;\pol{m-1}{n+1})$ so that $\int_0^1h(x;v,w)\der x=0$ for all $(v,w)\in\R^{n+1}$ and $f=\pi^*h+\der g$, where $\pi\colon M\to[0,1]$ is the projection. In addition,
\begin{equation}
\label{eq:torus-thm-estimate}
\aabs{h}_{L^2([0,1],\pol{m}{n+1})}+\aabs{g}_{H^1_0(M;\pol{m-1}{n+1})}
\leq
C
\aabs{f}_{L^2(M,\pol{m}{n+1})}.
\end{equation}
\end{enumerate}
\end{theorem}

Here~$H^1_0$ is the usual Sobolev space with zero boundary trace.
In the proof we will make use of the space~$H^1(\T^{1+n})$ and denote it by~$H^1(M)$ after identifying $[0,1]$ with~$\T^1$.
This~$H^1(M)$ is not the same as the usual~$H^1(M)$, but the space~$H^1_0(M)$ appearing in the theorem is the usual space of~$H^1$ functions on~$M$ vanishing at the boundary.

If $n=0$, we have simply $M=[0,1]$ and the result is trivial, so we may assume $n\geq1$ in the proof.
Since $\pol{-1}{n+1}=0$, there is no potential~$g$ (formally $g=0$) in the case of scalar functions ($m=0$) as expected.

We will develop the needed tools in the following subsections and prove theorem~\ref{thm:torus} in section~\ref{sec:torus-pf}.
But before embarking on the proof, we make a remark about the functions~$h$ and~$g$ used to describe the kernel.

\begin{remark}
The functions~$h$ and~$g$ are not unique; some of~$\pi^*h$ and~$\der g$ can be interchanged and additive constants (polynomials independent of the base point) in~$g$ have no effect on~$\der g$.
However, estimate~\eqref{eq:torus-thm-estimate} does not hold for arbitrary choices of~$h$ and~$g$.
There is a natural way to choose~$h$ and~$g$ in a unique way, and they will satisfy the estimate.

A calculation shows that the tensor field~$h$ also comes from a potential if and only if $h(x;0,w)=0$ for all $x\in M$ and $w\in\R^n$.
It follows that the X-ray transform of tensor fields of order $m\geq1$ is not solenoidally injective on~$M$ unless $n=0$.
On the manifold $[0,1]$ the scalar X-ray transform is non-injective, but tensor transforms of all orders are indeed solenoidally injective.
This is also true on~$\T^1$.
\end{remark}

\subsection{Weak definition of the X-ray transform}

For $\psi\in C^\infty(\T^n)$ and $b\in\R^n$ we define the extension $E_b\psi\in C^\infty(M)$ so that
\begin{equation}
E_b\psi(x,y)
=
\psi(y-bx).
\end{equation}
For $f\in L^2(M;\pol{m}{n+1})$ the integral $I^mf(a,b)$ is defined for every $b\in\R^n$ and almost every $a\in\T^n$, but it will be convenient to use the following weak formulation.

We define~$I^mf$ so that for any $b\in\R^n$ and $\psi\in C^\infty(\T^n)$
\begin{equation}
\label{eq:weak-torus-def}
\ip{I^mf(\dummy,b)}{\psi}_{\T^n}
=
\int_Mf(x,y;1,b)E_b\psi(x,y)\der x\der y.
\end{equation}
It is easy to check that this coincides with the more straightforward definition, and it also allows extending the definition to distributions if needed.

\subsection{Fourier series}

We denote $e_k(z)=e^{2\pi ik\cdot z}$ when $k\in\Z^l$ and $z\in\T^l$; the dimensions are left implicit as they are can be inferred from context.

We write $f\in L^2(M;\pol{m}{n+1})$ as Fourier series:
\begin{equation}
\label{eq:torus-series}
f(x,y;v,w)
=
\sum_{j\in\Z}
\sum_{k\in\Z^n}
\hat f(j,k;v,w)e_j(x)e_k(y),
\end{equation}
where the series converges in~$L^2(M)$ for any fixed~$(v,w)$.
We will denote the function $(x,y)\mapsto e_j(x)e_k(y)$ by $e_j\otimes e_k$.

We have identified~$[0,1]$ with~$\T^1$.
This makes no difference for~$L^2$ functions, but it will have an effect on~$H^1$ functions.

Using the Fourier series we can easily define the Sobolev spaces $H^s(M;\pol{m}{n+1})$ for any $s\in\R$ using the norms
\begin{equation}
\aabs{f}_{H^s(M;\pol{m}{n+1})}^2
=
\sum_{j,k}
(1+j^2+\abs{k}^2)^{s}
\int_{S^n}\abs{\hat f(j,k;v,w)}^2\der S(v,w),
\end{equation}
where~$S$ is the usual measure on $S^n\subset\R^{n+1}$.
Notice that this gives the standard Sobolev space on~$\T^{n+1}$; the periodic extension from~$[0,1]$ to~$\T^1$ means that these are not the usual~$H^s$ spaces on~$M$ unless $s=0$.

\begin{lemma}
\label{lma:torus-fourier}
If $f\in L^2(M;\pol{m}{n+1})$ and $I^mf=0$, then
\begin{equation}
\sum_{j\in\Z}\hat f(j,k;1,b)\phi(j+k\cdot b)
=
0
\end{equation}
for all $k\in\Z^n$ and $b\in\R^n$, where
\begin{equation}
\label{eq:phi-def}
\phi(t)
\coloneqq
\frac{e^{2\pi it}-1}{2\pi it}
\end{equation}
and $\phi(0)=1$.
\end{lemma}

\begin{proof}
Using the test function~$e_{-k'}$ in the weak definition~\eqref{eq:weak-torus-def} gives
\begin{equation}
\ip{I^0(e_j\otimes e_k)(\dummy,b)}{e_{-k'}}_{\T^n}
=
\delta_{k,k'}\phi(j+k'\cdot b).
\end{equation}
Using this together with the convergence of the series~\eqref{eq:torus-series} gives
\begin{equation}
\ip{I^m f(\dummy,b)}{e_{-k}}_{\T^n}
=
\sum_{j\in\Z}\hat f(j,k;1,b)\phi(j+k\cdot b).
\end{equation}
This proves the claim.
\end{proof}

\subsection{Properties of polynomials}

The lemmas of this section are mostly concerned with the properties of tensor fields that follow directly from the properties of polynomials.

\begin{lemma}
\label{lma:polynomial-factor}
Fix any $\xi\in\R^n\setminus0$.
If a homogeneous polynomial $F\colon\R^n\to\C$ of order~$m$ satisfies $F(u)=0$ whenever $u\cdot\xi=0$, then there is a homogeneous polynomial~$G$ of order $m-1$ so that $F(u)=u\cdot\xi G(u)$.
If $m=0$, then $G=0$.
\end{lemma}

The proof of the preceding lemma is quite elementary, and one can be found in~\cite[Lemma~11]{I:torus}.

\begin{lemma}
\label{lma:torus-Hs-potential}
Fix $m\in\N$ and $s\in\R$.
If $f\in H^s(M;\pol{m}{n+1})$ satisfies
\begin{equation}
\label{eq:torus-fourier-gradient}
\hat f(j,k;v,w)=2\pi i(jv+k\cdot w)\hat g(j,k;v,w)
\end{equation}
for some function $\hat g\colon\Z\times\Z^n\to\pol{m-1}{n+1}$, then~$g$ is the Fourier series of a function $g\in H^{s+1}(M;\pol{m-1}{n+1})$ which satisfies
\begin{equation}
\label{eq:torus-weak-gradient}
f(x,y;v,w)
=
(v\partial_x+w\cdot\nabla_y)g(x,y;v,w)
\end{equation}
in the weak sense.

Moreover, there is a constant $C=C(n,m)$ so that
\begin{equation}
\label{eq:torus-sobolev-estimate}
\aabs{g}_{H^{s+1}(M;\pol{m-1}{n+1})}
\leq
C
\aabs{f}_{H^{s}(M;\pol{m}{n+1})}
\end{equation}
whenever $\hat g(0,0;\dummy,\dummy)=0$.
\end{lemma}

In the language of tensor fields,~\eqref{eq:torus-weak-gradient} means $f=\der g$.

\begin{proof}[Proof of lemma~\ref{lma:torus-Hs-potential}]
In this proof the letter~$C$ can stand for different constants in different estimates.
Its dependence on various parameters is indicated in parentheses.

We equip the spaces~$\pol{m}{n+1}$ and~$\pol{m-1}{n+1}$ with the norm given by the natural embedding into~$L^2(S^n)$.
The spaces are finite dimensional so the choice of norms is mostly irrelevant, but this one is most convenient for us.
We assume $m\geq1$; in the case $m=0$ the function~$g$ is necessarily zero.

For any $\xi\in\R^{n+1}$, consider the operator $\mu_\xi\colon \pol{m-1}{n+1}\to\pol{m}{n+1}$ given by
\begin{equation}
(\mu_\xi p)(u)
=
(u\cdot\xi) p(u).
\end{equation}
This operator is injective for $\xi\neq0$ and the polynomial spaces are finite dimensional, so
\begin{equation}
\aabs{p}_{\pol{m-1}{n+1}}
\leq
C(n,m,\xi)
\aabs{\mu_\xi p}_{\pol{m}{n+1}}
\end{equation}
for any $\xi\in\R^{n+1}\setminus 0$.

The norms on the polynomial spaces are rotation invariant, so the direction of~$\xi$ is irrelevant.
Since also $\xi\mapsto\mu_\xi$ is linear, we have
\begin{equation}
\aabs{p}_{\pol{m-1}{n+1}}
\leq
C(n,m)
\abs{\xi}^{-1}
\aabs{\mu_\xi p}_{\pol{m}{n+1}}
\end{equation}
for all $\xi\in\R^{n+1}\setminus 0$.
This leads to
\begin{equation}
\label{eq:torus-polynomial-estimate}
\aabs{p}_{\pol{m-1}{n+1}}
\leq
C(n,m)
\left(1+\abs{\xi}^2\right)^{-1/2}
\aabs{\mu_\xi p}_{\pol{m}{n+1}}
\end{equation}
for all $\xi\in\Z^{n+1}\setminus 0$.

We abbreviate $l=(j,k)$ and $u=(v,w)$.
We have
\begin{equation}
\hat f(l;u)=2\pi i(\mu_l\hat g)(l;u).
\end{equation}
Integrating this over $u\in S^n$ with fixed~$l$ and using~\eqref{eq:torus-polynomial-estimate} gives
\begin{equation}
\label{eq:torus-sphere-estimate}
\left(1+\abs{l}^2\right)
\int_{S^n}\abs{\hat g(l;u)}^2\der S(u)
\leq
C(n,m)
\int_{S^n}\abs{\hat f(l;u)}^2\der S(u)
\end{equation}
for all $l\in\Z^{n+1}$, provided that $\hat g(0;\dummy)=0$.
Changing the polynomial $\hat g(0;\dummy)$ has no effect on the regularity of~$\hat g$ or~$g$.

Summing the estimate~\eqref{eq:torus-sphere-estimate} over $l\in\Z^{n+1}$ proves~\eqref{eq:torus-sobolev-estimate}.
From this estimate it follows that~$\hat g$ is the Fourier series of $g\in H^{s+1}(M;\pol{m-1}{n+1})$.
The identity~\eqref{eq:torus-weak-gradient} is the simply the Fourier transform of the assumed identity~\eqref{eq:torus-fourier-gradient} and therefore holds true in the weak sense.
\end{proof}

\begin{remark}
Lemma~\ref{lma:torus-Hs-potential} gives a simpler proof of the regularity result used for tensor tomography on tori~\cite[Lemma~12]{I:torus}.
\end{remark}

\subsection{Orbits, traces, and translations}

Any element $b\in\T^n$ induces a natural translation operator $\tau_b\colon\T^n\to\T^n$ by $\tau_b(z)=z+b$.
The orbit of~$b$ is the set of all its integer multiples, $\Z b\subset\T^n$.
We may write the element in terms of its components $b=(b_1,\dots,b_n)\in(\T^1)^n$.

\begin{lemma}
\label{lma:torus-dense-orbit}
The orbit of $b\in\T^n$ is dense on~$\T^n$ if and only if the numbers $1,b_1,\dots,b_n$ are linearly independent over the rationals.
The points~$b$ for which the orbit is dense are dense on~$\T^n$.
\end{lemma}

\begin{proof}
Let us first point out that the conditions are independent of translating the numbers~$b_i$ by integers, so the statement is independent of the choice of representative from~$\R^n$ for $b\in\T^n=\R^n/\Z^n$.

It follows from Weyl's equidistribution theorem that the orbit is dense if and only if
\begin{equation}
\lim_{N\to\infty}\frac1N\sum_{l=1}^Ne_k(lb)=0
\end{equation}
for all $k\in\Z^n\setminus0$.
This can in turn be confirmed to be equivalent with the linear independence statement of the claim.

Density of such points is an easy observation.
\end{proof}

\begin{lemma}
\label{lma:torus-potential-xrt}
Suppose $g\in H^1(M;\pol{m-1}{n+1})$, fix $b\in\R^n$, and denote $G(a)\coloneqq I^m(\der g)(a,b)$.
Then $G\in H^{1/2}(\T^n)$ and $G(a)=g(0,a+b;1,b)-g(0,a;1,b)$ for almost all $a\in\T^n$.
In particular, $G=0$ if and only if $g(0,\dummy;1,b)$ is invariant under~$\tau_b$.
\end{lemma}

\begin{proof}
Let us take any $\psi\in C^\infty(\T^n)$ and write $D_b\coloneqq\partial_x+b\cdot\nabla_y$.
This differential operator satisfies $D_bE_b\psi=0$.

Since~$g$ is $H^1$-regular as a function on~$\T^{n+1}$, its restrictions to the subtori $\{0\}\times\T^n$ and $\{1\}\times\T^n$ coincide and are $H^{1/2}$-regular.

Using the weak formulation of the X-ray transform given in~\eqref{eq:weak-torus-def}, we find
\begin{equation}
\begin{split}
\ip{G}{\psi}_{\T^n}
&=
\ip{I^m(\der g)(\dummy,b)}{\psi}_{\T^n}
\\&=
\int_M D_bg(x,y;1,b)E_b\psi(x,y)\der x\der y
\\&=
\int_{\T^n} g(1,y;1,b)\psi(y-b)\der y
-
\int_{\T^n} g(0,y;1,b)\psi(y)\der y
\\&=
\ip{\tau_bg(0,\dummy;1,b)-g(0,\dummy;1,b)}{\psi}_{\T^n}
.
\end{split}
\end{equation}
Since this holds for all $\psi\in C^\infty$, we have indeed $G(a)=g(0,a+b;1,b)-g(0,a;1,b)$ for almost all $a\in\T^n$.
\end{proof}

\subsection{Proof of theorem~\ref{thm:torus}}
\label{sec:torus-pf}

We are now ready to prove theorem~\ref{thm:torus} using the lemmas presented above.

\begin{proof}[Proof of theorem~\ref{thm:torus}]
It is clear that the second condition implies the first one; filling in the details in our formalism is an easy exercise.
Let us prove the converse.
Suppose that $f\in L^2(M;\pol{m}{n+1})$ satisfies $I^mf=0$.

We will again use the shorthands $l=(j,k)$ and $u=(v,w)$ when convenient.

For any $k\in\Z^n$ and $b\in\R^n$ we have by lemma~\ref{lma:torus-fourier}
\begin{equation}
\sum_{j'\in\Z}\hat f(j',k;1,b)\phi(j'+k\cdot b)
=
0.
\end{equation}
The sum is simple for $b\in\Z^n$, since $\phi(0)=1$ and~$\phi$ vanishes at other integers.

First, setting $k=0$ gives $\hat f(0,0;1,b)=0$ for all $b\in\R^n$.
Since~$\hat f(0;u)$ is homogeneous in~$u$, it therefore has to vanish for all $u\in\R^{n+1}$.

We also see that there is no information about~$f(j,0;u)$ for $j\neq0$.
We define $h\in L^2([0,1];\pol{m}{n+1})$ so that $\hat h(0;u)=0$ and $\hat h(j;u)=f(j,0;u)$ for all $j\neq0$.

Suppose then that $k\neq0$.
If $j+k\cdot b=0$, we get $\hat f(j,k;1,b)=0$.
Since $\hat f(l;u)$ is homogeneous in~$u$, this means that $\hat f(l;u)=0$ whenever $l\cdot u=0$ and $v\neq0$.

Let us denote $H_l=\{u\in\R^{n+1};l\cdot u=0\}$.
We thus know that the polynomial function $u\mapsto\hat f(l;u)$ vanishes whenever $u\in H_l\setminus H_{(1,0)}$.
Now~$l$ is not parallel to~$(1,0)$ since $k\neq0$, so the closure of $H_l\setminus H_{(1,0)}$ is the hyperplane~$H_l$.
Thus the polynomial vanishes in all of~$H_l$.

We have thus found that $\hat f(l;u)=0$ whenever $l\cdot u=0$ and $k\neq0$.
We then apply lemma~\ref{lma:polynomial-factor} to each~$l$ with $k\neq0$ to produce a polynomial~$\hat g(l;\dummy)$ with $\hat f(l;u)=2\pi i (l\cdot u)\hat g(l;u)$ for all $u\in\R^{n+1}$.
For other values of~$l$ we set $\hat g(l;\dummy)=0$.

Thus we have found a function $\hat g\colon\Z\times\Z^n\to\pol{m-1}{n+1}$ so that $\hat g(j,0;\dummy)=0$ for all $j\in\Z$ and $\hat f(j,k;v,w)=2\pi i(jv+k\cdot w)\hat g(j,k;v,w)$ whenever $k\neq0$.
If we denote $F=f-\pi^*h$, we may apply lemma~\ref{lma:torus-Hs-potential} to~$\hat F$, $\hat g$, and $s=0$.
We conclude that~$\hat g$ is the Fourier series of $g\in H^1(M;\pol{m-1}{n+1})$ which satisfies $F=\der g$.

Recall that~$g$ is an~$H^1$ function $\T^{n+1}\to\pol{m-1}{n+1}$ when~$[0,1]$ is identified with~$\T^1$, so the values of~$g$ at $x=0$ and $x=1$ coincide in the Sobolev sense.

Since $I^m(\pi^*h)=0$, it remains to show that $I^m(\der g)=0$ leads to $g\in H^1_0(M;\pol{m-1}{n+1})$.
To this end, it suffices to show that the function $G(y,b)\coloneqq g(0,y;1,b)$ is independent of $y\in\T^n$ for every $b\in\R^n$.
Because $\hat g(j,0;\dummy)=0$ for all $j\in\Z$, the only possible constant value is zero.

By lemma~\ref{lma:torus-potential-xrt} the function~$G(\dummy,b)$ is in~$H^{1/2}(\T^n)$ and is invariant under the action of~$\tau_b$.
Suppose $b\in(0,1)^n$ is such that the numbers $1,b_1,\dots,b_n$ are linearly independent over the rationals.
Then the orbit of~$b$ is dense in~$\T^n$ by lemma~\ref{lma:torus-dense-orbit} and~$G(\dummy,b)$ is invariant under the action of a dense subset of the whole translation group of~$\T^n$.
For any fixed $u\in L^2(\T^n)$ the map $\T^n\ni b\mapsto\tau_bu\in L^2(\T^n)$ is continuous.
This implies that~$G(\dummy,b)$ must in fact be constant.

The partial functions $b\mapsto G(y,b)$ for fixed~$y$ are polynomials of order $m-1$ but not necessarily homogeneous.
Therefore the partial functions~$G(\dummy,b)$ for a dense set of parameters~$b$ (see lemma~\ref{lma:torus-dense-orbit} for the density of parameters) fully determine the full function~$G$.
This implies that~$G(\dummy,b)$ is in fact a constant function for every $b\in\R^n$.

As explained above, this shows $g\in H^1_0$, and we have found the functions~$h$ and~$g$ we need.
Estimating~$h$ by~$f$ in~$L^2$ is trivial, and the estimate for~$g$ follows from~\eqref{eq:torus-sobolev-estimate}.
\end{proof}

\section{Tensor tomography on M\"obius strips and other twisted slabs}
\label{sec:twisted}

It is fairly straightforward to generalize our results from the strip $[0,1]\times\T^1$ to the M\"obius strip, and the same method works for various twisted versions of $[0,1]\times\T^n$.
We begin by setting up the necessary machinery.

\subsection{Covering spaces and deck transformations}

Let us begin with defining the class of manifolds we are working with.

\begin{definition}
Let $M=[0,1]\times\T^n$ and equip it with the usual Euclidean metric~$e$.
A twisted slab of dimension $1+n$ is a Riemannian manifold $(N,g_N)$ for which there is a smooth covering map $p\colon M\to N$ so that $p^*g_N=e$.
\end{definition}

The pullbacks over~$p$ are defined in the usual way.
Pushforwards are defined as deck averages as we will explain next.

The deck transformation group~$\Aut(p)$ is the group of diffeomorphisms $\phi\colon M\to M$ for which $p\circ\phi=p$.
This is a finite group and acts freely on~$M$.
For a scalar function $f\colon M\to\R$ the pushforward $p_*f\colon N\to\R$ is defined as
\begin{equation}
p_*f(x)
=
\frac1{\abs{\Aut(p)}}\sum_{y\in p^{-1}(x)}f(y).
\end{equation}
This definition can be extended to tensor fields and other objects using local diffeomorphisms given by restrictions of~$p$ to small open sets.

The pushforward~$p_*$ is a left inverse of the pullback~$p^*$.
There is no right inverse as~$p^*$ is typically not surjective.

The simplest non-trivial example of a twisted slab is the M\"obius strip, which has a two-fold cover by the usual strip $[0,1]\times\T^1$.
In higher dimensions there are more exotic ways to glue together faces of $[0,1]\times[0,1]^n$ to produce twisted slabs.
A twisted slab is necessarily compact.


The twisted slabs can also be stretched; see remark~\ref{rmk:stretch}.

\subsection{Tensor tomography in twisted slabs}

Now we are ready to state and prove our tensor tomography result for twisted slabs.
In dimension one the only twisted slab is~$[0,1]$, and the theorem adds nothing new to theorem~\ref{thm:torus}.
Twisted slabs exist in all higher dimensions.

\begin{theorem}
\label{thm:twisted}
Let~$N$ be a twisted slab of dimension two or higher, and fix any integer $m\geq0$.
Let~$f$ be an $L^2$-regular tensor field on~$N$.
The following are equivalent:
\begin{enumerate}
	\item The X-ray transform of~$f$ vanishes: $I^mf=0$.
	\item
There is $h\in L^2([0,1],\pol{m}{n+1})$ and a tensor field~$g$ on~$N$ of order $m-1$ with coefficients in~$H^1_0$ so that $\int_0^1h(x;v,w)\der x=0$ for all $(v,w)\in\R^{n+1}$ and $f=p_*\pi_M^*h+\der g$, where $\pi_M\colon M\to[0,1]$ is the projection to the first component.
\end{enumerate}
In addition, there is an estimate like~\eqref{eq:torus-thm-estimate} in theorem~\ref{thm:torus}.
\end{theorem}

\begin{remark}
If~$N$ has only one boundary component, the term~$p_*\pi_M^*h$ is symmetric under the flip across the ``equator'' $p(\{\frac12\}\times\T^n)\subset N$.
Therefore it would suffice to define~$h$ on~$[0,\frac12]$ instead of~$[0,1]$.
We prefer the present formulation, which works whether the number boundary components is one or two.

For a simple example of what the kernel looks like in the case of one boundary component, consider scalar tomography on the M\"obius strip ($m=0$, $n=1$).
The M\"obius strip can be identified with~$[0,1]^2$, where $\{0\}\times[0,1]$ and $\{1\}\times[0,1]$ are identified with opposite orientations.
(Or in fact $[0,\frac12]\times[0,1]$, but stretching is irrelevant; see remark~\ref{rmk:stretch}.)
Functions in the kernel are of the form $f(x,y)=h(x)$, where $\int_0^{1/2}h(x)\der x=0$ and $h(x)=h(1-x)$.
There is more symmetry than for the usual strip due to the orientation-flipping gluing.
\end{remark}

\begin{proof}[Proof of theorem~\ref{thm:twisted}]
First, it is easy to observe that the second condition implies the first.
Also, the estimate follows from theorem~\ref{thm:torus}.

For clarity, let us indicate the underlying manifold of the ray transforms by a subscript.
Suppose the tensor field~$f$ on~$N$ satisfies $I^m_Nf=0$.
Then also $I^m_M(p^*f)=0$.
Using theorem~\ref{thm:torus}, we find~$h$ and~$\tilde g$ so that $p^*f=\pi^*_Mh+\der\tilde g$.
Applying~$p_*$ we find that
\begin{equation}
f
=
p_*\pi^*_Mh+\der p_*\tilde g.
\end{equation}
We let $g=p_*\tilde g$, making the potential~$g$ on~$N$ the deck average of the potential~$\tilde g$ of the pullback~$p^*f$.
\end{proof}

\subsection*{Acknowledgements}

J.I.\ was supported by the Academy of Finland (decision 295853), and he is grateful for hospitality and support offered by the University of Washington during visits.
G.U.\ was partly supported by NSF.

\bibliographystyle{plain}
\bibliography{ip}

\end{document}